\documentclass{elsarticle}

\usepackage[all,cmtip]{xy}
\usepackage{array}
\usepackage{amsmath, amssymb}
\usepackage{graphics}
\usepackage{caption}
\usepackage{subcaption}
\usepackage{amsthm}
\usepackage{algorithm}
\usepackage{algorithmic}
\usepackage{booktabs}
\usepackage{color}
\usepackage{cancel}
\usepackage{caption}
\usepackage{subcaption}
\usepackage{pdflscape}
\usepackage{epstopdf}
\usepackage[normalem]{ulem}

\newtheorem{theorem}{{\bf Theorem}}
\newtheorem{remark}{{\bf Remark}}

\newtheorem{corollary}[theorem]{{\bf Corollary}}
\newtheorem{proposition}[theorem]{{\bf Proposition}}
\newtheorem{lemma}[theorem]{{\bf Lemma}}
\newtheorem{example}{{\bf Example}}

\def\bfx{\boldsymbol{x}}

\def\bfw{\boldsymbol{w}}

\def\bfQ{{\boldsymbol{Q}}}
\def\bfP{{\boldsymbol{P}}}

\begin{document}

\begin{frontmatter}
\title{Algebraic surfaces invariant under non-Euclidean affine rotations }

\author[a]{Juan G. Alc\'azar}
\ead{juange.alcazar@uah.es}
\author[b]{Ron Goldman}
\ead{rng@rice.edu}


\address[a]{Departamento de F\'{\i}sica y Matem\'aticas, Universidad de Alcal\'a,
E-28871 Madrid, Spain}
\address[b]{Department of Computer Science, Rice University, Houston, Texas 77005, USA}



\begin{abstract}
\emph{Affine rotation surfaces} are a generalization of the well-known surfaces of revolution. Affine rotation surfaces arise naturally within the framework of {\it affine differential geometry}, a field started by Blaschke in the first decades of the past century. Affine rotations are the affine equivalents of Euclidean rotations, and include certain shears as well as Euclidean rotations. Affine rotation surfaces are surfaces invariant under affine rotations. In this paper, we analyze several properties of algebraic affine rotation surfaces and, by using some notions and results from affine differential geometry, we develop an algorithm for determining whether or not an algebraic surface given in implicit form, or in some cases in rational parametric form, is an affine rotation surface. We also show how to find the axis of an affine rotation surface. Additionally, we discuss several properties of \emph{affine spheres}, analogues of Euclidean spheres in the context of affine differential geometry.
\end{abstract}

\end{frontmatter}

\section{Introduction.} \label{intro}

Classical differential geometry, initiated by Gauss in the first decades of the $19^{\mbox{th}}$ century, studies Euclidean invariants -- normal vectors and normal lines, mean curvature and Gaussian curvature -- under rigid motions. In contrast, \emph{affine differential geometry}, initiated by Blaschke in the first decades of the $20^{\mbox{th}}$ century, studies the corresponding affine invariants -- affine normals and affine curvatures -- which are invariant under the unimodular affine group generated by the \emph{special linear group} ${\bf SL}_3({\Bbb R})$, i.e. the group of matrices with determinant equal to 1. 

Classical surfaces of revolution are surfaces invariant under Euclidean rotations about a fixed axis. \emph{Affine rotation surfaces} are surfaces invariant under {\it affine rotations}, groups of linear transformations that leave a line in 3-space (the \emph{axis}) unchanged, and which include classical Euclidean rotations, and also certain classical and scissor shears (see Section 2). 

There are three types of affine rotation surfaces corresponding to the three types of affine rotation groups: elliptic rotation surfaces correspond to surfaces invariant under classical rotations; hyperbolic rotation surfaces correspond to surfaces invariant under hyperbolic rotations, also called scissor shears \cite{AGHM16}; and parabolic rotation surfaces correspond to surfaces invariant under transformations which are composites of certain classical shears. The elliptic affine rotation surfaces are the classical surfaces of revolution, which have been studied extensively in the algebraic case; see for example \cite{AG14, AG17}. The hyperbolic affine rotation surfaces are what the authors have previously called \emph{scissor shear invariant surfaces} \cite{AGHM16}, because these surfaces are invariant under certain scissor shears that leave an axis line unchanged. Algebraic surfaces of this type have been studied recently by the authors in \cite{AGHM16}. 

The goal of this paper is to extend some properties of classical surfaces of revolution to affine rotation surfaces, using some concepts and results from affine differential geometry. In order to carry out this investigation, we need to study several geometric questions related to parabolic affine rotation surfaces, and some properties of \emph{affine spheres}, surfaces whose affine normal lines intersect at a single point. 

Although the notion of an affine sphere appeared first in the context of affine differential geometry, affine spheres are also related to other fields in Mathematics, including real Monge-Amp\'ere equations, projective structures on manifolds, and the geometry of Calabi-Yau manifolds. The interested reader can consult the surveys \cite{Loftin, Loftin-2} for further information. Interestingly, even though affine spheres are not necessarily affine rotation surfaces, we show that there are certain connections between affine spheres and affine rotation surfaces. 

Our three main contributions are to provide:

\begin{itemize}
\item [1.] A geometric characterization of algebraic affine rotation surfaces.
\item [2.] Several geometric properties of parabolic algebraic affine rotation surfaces and algebraic affine spheres. 
\item [3.] An algorithm for detecting whether or not an algebraic surface given in implicit algebraic form, or in some cases in rational parametric form, is an affine rotation surface and for finding the axis lines of these affine rotation surfaces.
\end{itemize}

Each of these problems has been studied for classical surfaces of revolution \cite{AG14, AG17, Vrseck}. One of the main analysis techniques is to use the fact that for classical surfaces of revolution the Euclidean normal lines all intersect the axis of rotation. But for affine rotation surfaces, the Euclidean normal lines no longer intersect the axis line. The failure of this property makes extending results from surfaces of revolution to affine rotation surfaces using standard Euclidean differential geometry quite difficult \cite{AGHM16}. \emph{The key fact that makes the analysis of affine rotation surfaces tractable is that for affine rotation surfaces the affine normal lines all intersect the axis line}. It is primarily this affine analogue of the classical result about the normal lines of a classical surface of revolution that allows us to
more readily extend results for classical surfaces of revolution to affine rotation surfaces. Thus affine differential geometry is one of the keys to understanding and analyzing affine rotation surfaces.

The three types of affine rotation surfaces -- elliptic, hyperbolic, and parabolic -- share several geometric properties: they all have a fixed axis line, and their cross sections
perpendicular to a fixed direction are conic sections (circles, hyperbolas, or parabolas). For elliptic and hyperbolic affine rotation surfaces, these cross sections are by planes
perpendicular to the axis line, but for parabolic affine rotation surfaces these cross sections are by planes parallel to the axis line. This distinction makes the investigation of parabolic affine rotation surfaces a bit more challenging. In this paper we derive several results concerning parabolic affine rotation surfaces; nevertheless, many of our theorems and proofs extend readily and sometimes more easily to the other two types of affine rotation surfaces.

This paper is organized in the following fashion. In Section 2 we formally define what we mean by affine rotations and affine rotation surfaces. We also recall some key notions and formulas from affine differential geometry, including the affine co-normal vector, the affine normal vector, and the affine normal line. In addition, we give a preliminary characterization of affine rotation surfaces in terms of the affine normal lines and the shadow line property (see Theorem \ref{th-imp}). Section 3.1 is devoted to the study of affine spheres; Section 3.2 is devoted to the investigation of parabolic affine rotation surfaces, but similar properties often extend readily to elliptical and hyperbolic affine rotation surfaces (see \cite{AG14, AG17, AGHM16}). We present Section 3.1 and Section 3.2 together in Section 3, because some results in Section 3.2 provide certain conditions under which a parabolic affine rotation surface is an affine sphere. In Section 4, we present an algorithm using Pl\"ucker coordinates for finding the axes of affine rotation of an affine rotation surface given in implicit algebraic form, and for detecting whether or not an algebraic surface is an affine rotation surface; we also address some cases of surfaces given in rational parametric form. We close in Section 5 with a brief summary of our work along with several open problems for future research. Two appendixes are also provided: the first to elucidate the geometry of certain cones that appear in the analysis of parabolic affine rotation surfaces, and the second with more complete descriptions of the algorithms for detecting whether or not a surface in implicit algebraic or in rational parametric form is an affine rotation surface and for finding the axes of affine rotation surfaces.

\section{Preliminaries.} \label{sec-prelim}
\subsection{Affine rotation surfaces.} \label{prelim-new}

An \emph{affine rotation group} is a uniparametric matrix group that is a subgroup of the special linear group ${\bf SL}_3({\Bbb R})$, i.e. the group of matrices with determinant equal to 1, which leaves invariant exactly one line of 3-space, called the \emph{affine axis of rotation}. Lee \cite{Lee} shows that there are only three different types of such subgroups; in an appropriate coordinate system, these types correspond to the following uniparametric matrix groups: 

\begin{equation}\label{types}
\begin{array}{ccc}
\begin{pmatrix}\cos(\alpha) & -\sin(\alpha) & 0 \\ \sin(\alpha) & \cos(\alpha) & 0 \\ 0 & 0 & 1\end{pmatrix}, & 
\left(\begin{array}{ccc}
\mbox{cosh}(\alpha) & \mbox{sinh}(\alpha) & 0\\
\mbox{sinh}(\alpha) & \mbox{cosh}(\alpha) & 0\\
0 & 0 & 1
\end{array}\right), & 
\begin{array}{cc}
\left(\begin{array}{ccc}
1 & 0 & 0 \\
\alpha & 1 & 0 \\
\frac{\alpha^2}{2} & \alpha & 1
\end{array}\right).
\end{array}
\end{array}
\end{equation}

In the three cases of Eq. \eqref{types}, the invariant line is the $z$-axis. We name the rotations defined in each case as \emph{elliptic} (left-most matrix, which defines a classical rotation about the $z$-axis), \emph{hyperbolic} (center matrix, which defines a hyperbolic rotation about the $z$-axis), and \emph{parabolic} (right-most matrix). 

In turn, the surfaces invariant under one of these matrix groups are called \emph{affine rotation surfaces}, and are said to be of elliptic, hyperbolic or parabolic type depending on the form of the matrix group. The affine rotation surfaces of elliptic type are the well-known surfaces of revolution. Thus, affine rotation surfaces are generalizations of surfaces of revolution. The affine rotation surfaces of hyperbolic type are studied in \cite{AGHM16}, where they are called \emph{scissor-shear invariant} surfaces, or {\sc ssi}-surfaces, for short. Fig. \ref{ex-fig} shows a parabolic affine rotation surface (left), and a hyperbolic affine rotation surface (right). 

\begin{figure}
\begin{center}
$$\begin{array}{cc}
\includegraphics[scale=0.33]{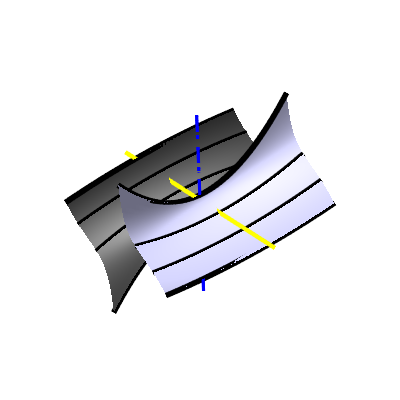} & \includegraphics[scale=0.33]{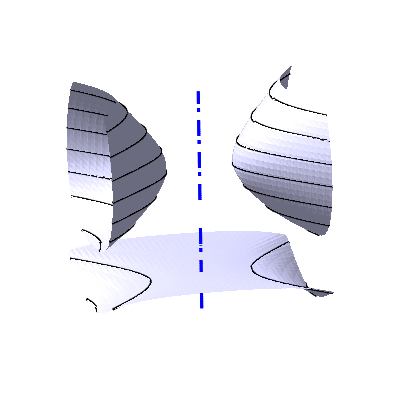}
\end{array}$$
\end{center}
\caption{Parabolic (left) and hyperbolic (right) affine rotation surfaces. The dotted blue lines are the affine axes of rotation. For the parabolic affine rotation surface, the yellow line is normal to all the planes containing the parallel curves (which are parabolas).}\label{ex-fig}
\end{figure} 

Every affine rotation surface about the $z$-axis can be parametrized locally around a regular point using differentiable functions $f(s),g(s)$ as 
\begin{equation}\label{local-param}
{\bf x}(\alpha, s)=\bfQ_{\alpha}\cdot [f(s),0,g(s)]^T,
\end{equation}
where $[f(s),0,g(s)]^T$ parametrizes a \emph{directrix curve} and $\bfQ_{\alpha}$ corresponds to one of the uniparametric matrix groups in Eq. \eqref{types}. Using this representation, the curves ${\bf x}(\alpha_0,s)$ are called \emph{meridians}, while the curves ${\bf x}(\alpha,s_0)$ are called \emph{parallel curves}. In particular, the directrix is a meridian. Moreover, one can show that the parallel curves are \cite{Lee} (a) in the elliptic case, circles centered on the $z$-axis, contained in planes normal to the $z$-axis (see Fig. \ref{ex-fig-parallel}, left); (b) in the hyperbolic case, rectangular hyperbolas centered on the $z$-axis, contained in planes normal to the $z$-axis, with the same center and asymptotes (see Fig. \ref{ex-fig}, right, and also Fig. \ref{ex-fig-parallel}, center); (c) in the parabolic case, parallel parabolas placed in planes normal to the $x$-axis, with the same axes of symmetry, where the axes of symmetry are parallel to the $z$-axis, i.e. to the axis of rotation (see Fig. \ref{ex-fig}, left, and also Fig. \ref{ex-fig-parallel}, right). 

\begin{figure}
\begin{center}
$$\begin{array}{ccc}
\hspace{-1.5 cm} \includegraphics[scale=0.33]{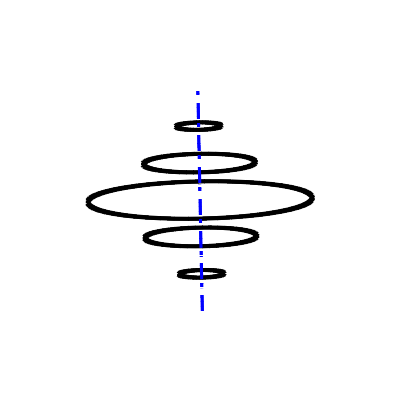}  & \includegraphics[scale=0.33]{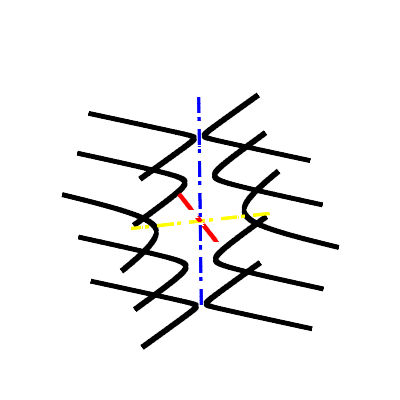} & \includegraphics[scale=0.33]{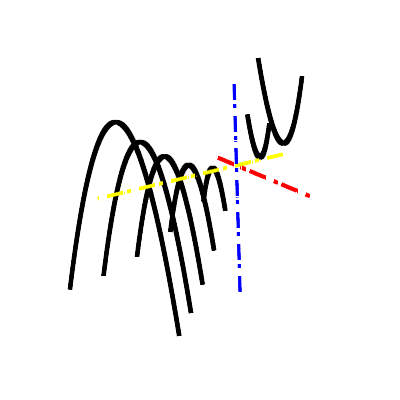} 
\end{array}$$
\end{center}
\caption{Parallel curves for each type of affine rotation surface: elliptic (left), hyperbolic (center), parabolic (right)}\label{ex-fig-parallel}
\end{figure}

Each of the matrix groups $\bfQ_{\alpha}$ in Eq. \eqref{types} preserves a certain quadratic form $p(\bfx)$ (see \cite{Lee}), in the sense that $p(\bfx)=p(\bfQ_{\alpha}\bfx)$: elliptic affine rotations about the $z$-axis preserve $p(x,y)=x^2+y^2$; hyperbolic affine rotations about the $z$-axis preserve $p(x,y)=x^2-y^2$; and parabolic affine rotations about the $z$-axis preserve $p(x,y,z)=y^2-2xz$. Because of this invariance, since elliptic and hyperbolic affine rotations about the $z$-axis preserve the $z$-coordinate, if $F(x,y,z)=0$ implicitly defines an algebraic surface $S$ invariant under such rotations, then $F(\bfx)=\tilde{F}(p(\bfx),z)$, where $\tilde{F}$ is a bivariate polynomial. However, parabolic affine rotations about the $z$-axis leave the $x$-coordinate, and not the $z$-coordinate, invariant, so if $F(x,y,z)=0$ implicitly defines an algebraic surface $S$ invariant under such an affine rotation, then $F(\bfx)=\tilde{F}(p(\bfx),x)$. We shall return to this key observation in Section \ref{subsec-parab} (see Theorem \ref{th-general}).

Observe also that in the elliptic case there is just one distinguished direction, namely that of the axis of rotation, because a surface of revolution is completely symmetric about the axis of rotation. But in the hyperbolic and parabolic cases there are three distinguished directions, because we need to make precise three orthogonal directions in order to define these affine rotation groups. In the hyperbolic case in addition to the direction of the affine rotation axis, we also need the directions corresponding to the major and minor axes of the parallel curves, i.e. the hyperbolas. In the parabolic case in addition to the direction of the affine rotation axis, we also need the direction normal to all the planes containing the parallel curves, i.e. the parabolas, and the direction normal to both this last direction and the direction of the affine rotation axis. Fig. \ref{ex-fig-parallel} illustrates these ideas. Here, the axis of rotation is plotted in blue; in the hyperbolic case (center), the direction of the major axis of the parallel curves is plotted in yellow; in the parabolic case (right), the direction normal to the planes containing the parallel curves is also plotted in yellow. The direction normal to both the blue and the yellow directions is plotted in red. 

\begin{remark}\label{rem-con} Conversely, if there exists a line ${\mathcal A}$ such that the intersections of a surface $S$ with planes $\Pi$ normal to ${\mathcal A}$ are circles centered at the points $\Pi\cap {\mathcal A}$ we can locally parametrize $S$ as in \eqref{local-param}, after a suitable orthogonal change of coordinates, and therefore we can recognize $S$ as an elliptic surface of rotation about ${\mathcal A}$. Similarly if the sections $\Pi\cap S$ are rectangular hyperbolas with parallel asymptotes, centered at the points $\Pi\cap {\mathcal A}$, $S$ is a hyperbolic surface of rotation about ${\mathcal A}$. However, the fact that there exists a line ${\mathcal A}$ and a family of parallel planes $\tilde{\Pi}$, \emph{parallel} to ${\mathcal A}$, such that $\tilde{\Pi}\cap S$ are parabolas whose axes of symmetry are lines through $\tilde{\Pi}\cap S$ parallel to ${\mathcal A}$, is not enough to guarantee that the surface is a parabolic affine rotation surface: a counterexample is the paraboloid of revolution. \end{remark}

An affine rotation surface can have more than one axis of affine rotation, i.e. that a surface can be an affine rotation surface in more than one way. For elliptic affine rotation surfaces, the sphere, which has infinitely many axes, is the only example. However, hyperbolic affine rotation surfaces can have either 1, or 3, or infinitely many axes of affine rotation (see Section 3.2 of \cite{AGHM16}). Furthermore, it is also possible that the same surface is an affine rotation surface of several types; for instance, the cone $x^2=y^2+z^2$ is an elliptic affine rotation surface about the $x$-axis, and a hyperbolic affine rotation surface about the $z$-axis. We will explore this question in more detail in Section \ref{subsec-parab}; the interested reader can also check Section 3 of \cite{AGHM16} for further reading on this question for the case of hyperbolic affine rotation surfaces.  

\subsection{A first characterization for affine rotation surfaces.}

In order to characterize affine rotation surfaces, we need to recall some notions from affine differential geometry that we take from \cite{Andrade,Freitas}. We introduce these concepts first for surfaces $S$ parametrized by ${\bf x}:U\subset {\Bbb R}^2\to {\Bbb R}^3$, where ${\bf x}={\bf x}(u,v)$ is a mapping with sufficiently good properties. Additionally, we require that $S$ has Gaussian curvature not identically equal to zero, i.e. that $S$ is not locally isometric to the plane, or, in other words, $S$ is not developable. 

The \emph{affine co-normal vector} at each point of $S$ is defined as
\begin{equation} \label{nu}
\nu =|K|^{-\frac{1}{4}}\cdot {\bf N},
\end{equation}
where ${\bf N}$ is the unitary Euclidean normal vector, and $K$ is the Gaussian curvature. The affine co-normal vector is not defined when $K$ is zero, which is the reason why we require $S$ to have nonzero Gaussian curvature. 

The \emph{affine normal vector} to $S$ at a point $p\in S$ is  
\begin{equation}\label{affine-normal}
\xi(p)=[\nu(p),\nu_u(p),\nu_v(p)]^{-1} \left(\nu_u(p)\times \nu_v(p)\right),
\end{equation}
where $\bullet_u,\bullet_v$ represent the partial derivatives of $\bullet$ with respect to the variables $u,v$, and $[\bullet,\bullet_u,\bullet_v]$ represents the determinant of $\bullet,\bullet_u,\bullet_v$. The affine normal vectors are known to be covariant under affine transformations (see Prop. 3 in \cite{Andrade}), i.e. if $h$ represents an affine transformation, then $\xi(h(p))=h(\xi(p))$. The {\it affine normal line} at $p\in S$ is the line through $p$, parallel to the affine normal vector. 

If the surface $S$ is given implicitly by $F(x,y,z)=0$, then by the Implicit Function Theorem the affine co-normal vector is given by \cite{AL}
\[
\nu=\frac{1}{d^{1/4}}\left(\frac{F_x}{F_z},\frac{F_y}{F_z},1\right),
\]
where $d$ is a rational function of $x,y,z$ and $F_x,F_y,F_z$ are the partial derivatives of $F$ with respect to $x,y,z$. An explicit expression for the affine normal vector $\xi$ for surfaces defined implicitly by $F(x,y,z)=0$ can also be derived; such an expression can be found in Appendix A of \cite{Andrade}. 

Before giving the characterization of affine rotation surfaces that will be key for us, we need to introduce one more property: we say that a surface $S$ has the \emph{shadow line} property with respect to a line ${\mathcal A}$, if along every meridian, i.e. intersections of $S$ with planes containing ${\mathcal A}$, the tangents to the parallel curves are parallel (so that ``enlightening the surface in the direction of these tangents produces the meridian as a {\it shadow line}" \cite{Man-mail}). 

\begin{lemma} \label{lem-shadows} Affine rotation surfaces have the shadow line property with respect to the axis line ${\mathcal A}$.
\end{lemma} 

\begin{proof} Without loss of generality, let us assume that $S$ can (at least locally) be parametrized as in Eq. \eqref{local-param}, where $\bfQ_{\alpha}$ represents a canonical rotation group, and ${\mathcal A}$ is the $z$-axis. We need to see that for all $s\in {\Bbb R}$, $\left.\frac{\partial {\bf x}}{\partial \alpha}\right\rvert_{\alpha=\alpha_0}=\mu(s) {\bf v}$, where ${\bf v}$ is constant for a fixed value $\alpha=\alpha_0$. In order to establish this result, one can check this statement separately for the elliptic, hyperbolic and parabolic cases. We consider just the parabolic case, and we leave the remaining two cases to the reader. In the parabolic case we have 
\[
\left.\frac{\partial {\bf x}}{\partial \alpha}\right\rvert_{\alpha=\alpha_0}=[0,f(s),\alpha_0 f(s)]^T=f(s)\cdot [0,1,\alpha_0]^T,\]
so that ${\bf v}=[0,1,\alpha_0]^T$, which is a constant vector whenever $\alpha_0$ is fixed. 
\end{proof}

Now we can give a characterization for affine rotation sufaces. 

\begin{theorem}\label{th-imp} The surface $S$ is an affine rotation surface with axis ${\mathcal A}$ if and only if the following two conditions hold: (1) All the affine normal lines intersect ${\mathcal A}$; (2) $S$ has the shadow line property with respect to the line ${\mathcal A}$. 
\end{theorem} 

\begin{proof} The implication $(\Leftarrow)$ is Theorem 18 in \cite{Su}. As for $(\Rightarrow)$, without loss of generality it suffices to check these two conditions for the surfaces parametrized as in Eq. \eqref{local-param}, where $\bfQ_{\alpha}$ represents a canonical affine rotation group. Condition (2) on the shadow line property is proved in Lemma \ref{lem-shadows}. Condition (1) can be checked by verifying that in all the cases, the affine normal lines of these surfaces intersect the $z$-axis; the calculations are lengthy, but are verified in \cite{impl1}.
\end{proof}

\begin{figure}
\begin{center}
$$\begin{array}{c}
\includegraphics[scale=0.4]{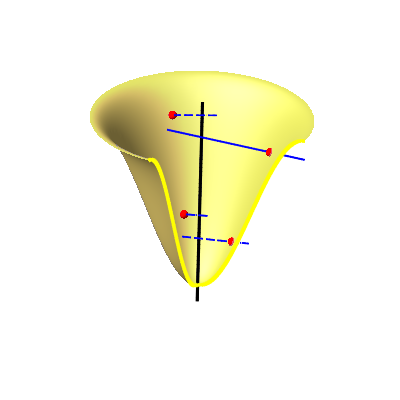}  
\end{array}$$
\end{center}
\caption{Affine normal lines (in blue) intersecting the affine axis of rotation (the thick black line) of an elliptic affine rotation surface.}\label{fig-intersect}
\end{figure}

Figure \ref{fig-intersect} illustrates Theorem \ref{th-imp} for an elliptic affine rotation surface. 

In Section \ref{subsec-aff-spheres} we will see that the implication $(\Leftarrow)$ of Theorem \ref{th-imp} really requires the shadow line property, i.e. it is not enough that all the affine normal lines intersect a line ${\mathcal A}$ to ensure that $S$ is an affine rotation surface about ${\mathcal A}$. Nevertheless, the characterization of affine rotation surfaces in Theorem \ref{th-imp} is not well suited for computational purposes, because it is not easy to derive from this characterization an algorithm for detecting whether or not a given surface is an affine rotation surface. For this reason, in Section \ref{sec-algebraic} when we study algebraic surfaces, we will derive other results, based on Theorem \ref{th-imp}, to determine whether or not a given surface is an affine rotation surface in a more efficient way. 

In the rest of this paper we will impose the following conditions on the surface $S$: 

\begin{itemize}
\item [(i)] $S$ is an \emph{algebraic} surface, i.e. $S$ is the zero set of a polynomial $F(x,y,z)$, which defines its \emph{implicit} equation.
\item [(ii)] $S$ is irreducible. 
\item [(iii)] $S$ is real, i.e. the real part of $S$ is 2-dimensional, so $S$ is not degenerate (as happens, for instance, with surfaces like $x^2+y^2+z^2=0$, or $x^2+y^2=0$).
\item [(iv)] The Gaussian curvature of $S$ is not identically zero (i.e. $S$ is not developable), so that the affine co-normal and the affine normal vectors are defined at almost all points of $S$. In particular, $S$ is not a plane. 
\end{itemize}

\section{ Affine spheres and parabolic affine rotation surfaces.}\label{new-parab-sec}

In this section we recall the notion, well-known in affine differential geometry, of an \emph{affine sphere}, a surface whose affine normals intersect at a single point. Additionally, we will study certain properties of parabolic, algebraic affine rotation surfaces; analogous properties for the elliptic and the hyperbolic cases are derived in \cite{AG14} and \cite{AGHM16}.  

\subsection{Affine spheres.}\label{subsec-aff-spheres}

A surface $S$ is said to be an \emph{affine sphere} if all the affine normals of $S$ intersect at a single point, called the \emph{center} of the sphere; the interested reader can check \cite{Loftin, Loftin-2, Nomizu} for further reading on this topic. Affine spheres are classified into three different types: \emph{elliptic}, if the center is an affine point lying on the convex side of $S$; \emph{hyperbolic}, if the center is an affine point lying on the concave side of $S$; and \emph{parabolic}, if the center is a point at infinity (i.e. if all the affine normals are parallel). Quadrics are the simplest non-trivial examples of affine spheres. Furthermore, any algebraic elliptic or parabolic affine sphere must be a quadric (see Section 3.3 of \cite{Loftin-2}). However, not all hyperbolic spheres are quadric surfaces: for instance, the cubic surface $xyz=1$ is an affine sphere (see Theorem 6 of \cite{Simon}). Figure \ref{fig-aff} shows two affine spheres, jointly with their affine normal lines (the solid black lines), the ellipsoid and the cubic $xyz=1$; in both cases the center is located at the origin.

\begin{figure}
\begin{center}
$$\begin{array}{cc}
\hspace{-2 cm}\includegraphics[scale=0.7]{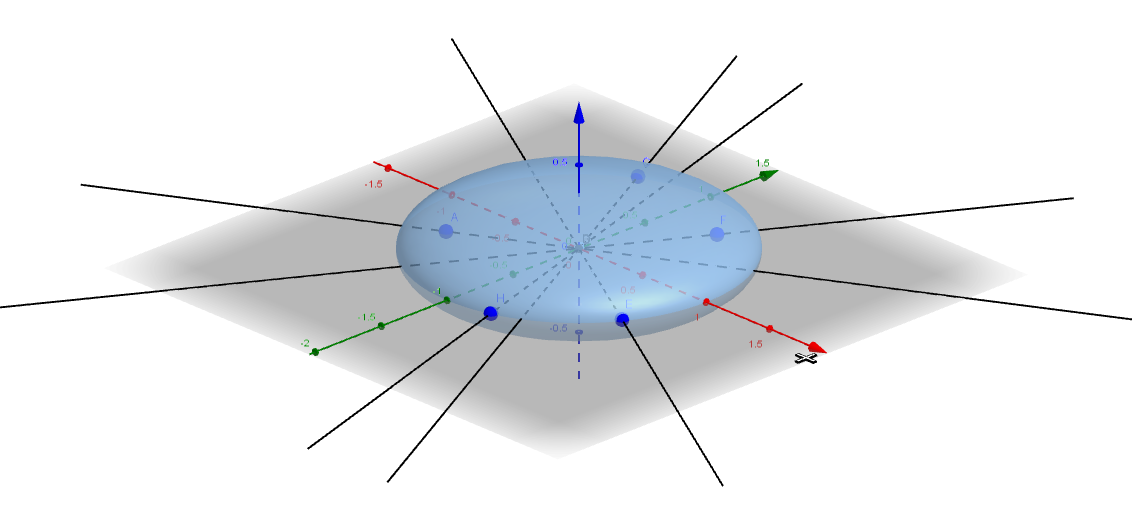} & \includegraphics[scale=0.7]{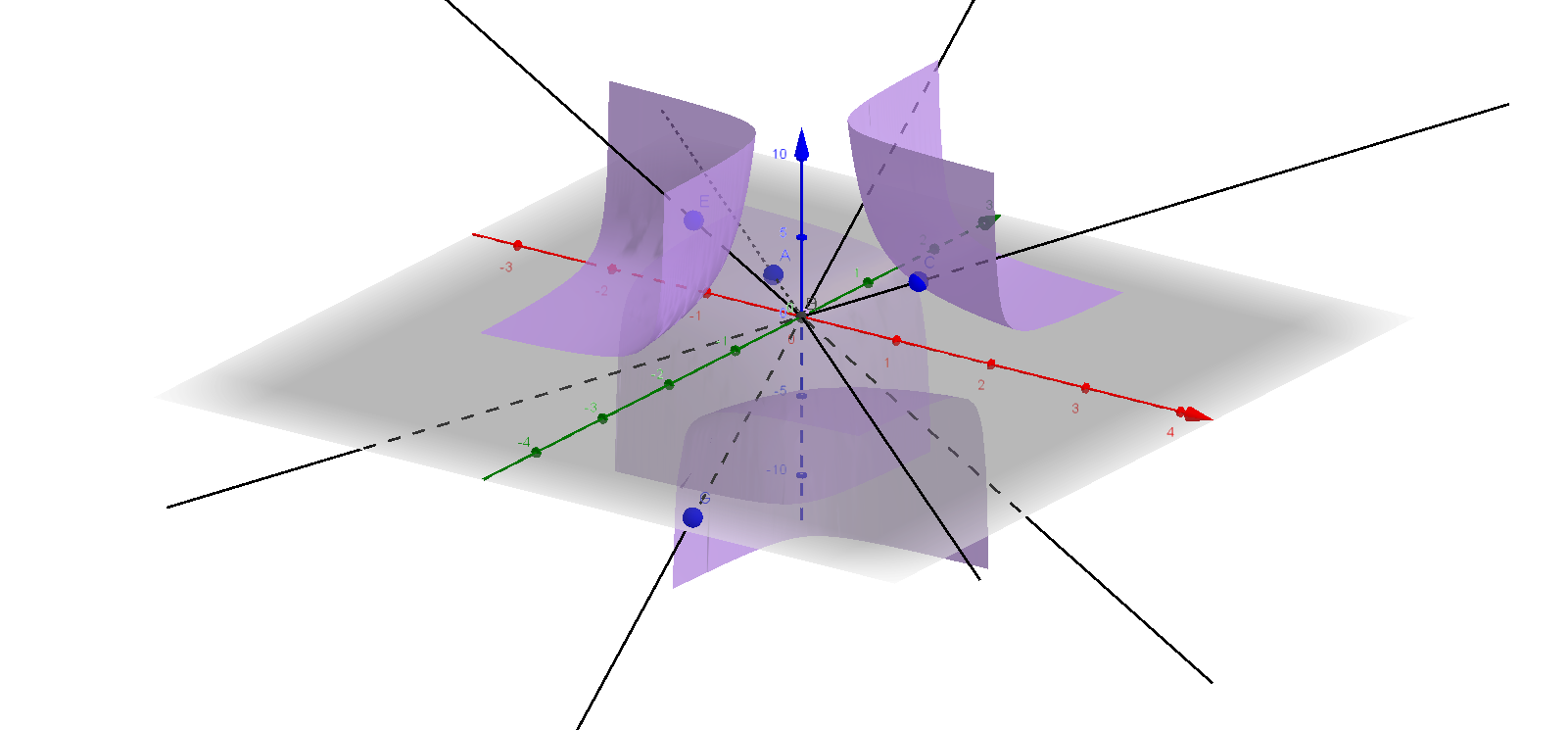}
\end{array}$$
\end{center}
\caption{Two affine spheres: the ellipsoid (left) and the cubic surface $xyz=1$ (right). The black lines are the affine normal lines.}\label{fig-aff}
\end{figure}

The term ``affine sphere" can be misleading, since affine spheres need not be affine rotation surfaces. For instance, an ellipsoid that is not a surface of revolution is an elliptic affine sphere, but not an affine rotation surface. Indeed, if an ellipsoid is not a surface of revolution, then such an ellipsoid is not an 
elliptic affine rotation surface. But such an ellipsoid cannot be a parabolic or hyperbolic affine rotation surface either, because there is no family of parallel planes whose intersection with an ellipsoid produces hyperbolas or parabolas. Interestingly, this example also shows that the shadow line property is necessary to prove the implication $(\Leftarrow)$ of Theorem \ref{th-imp}: otherwise, in the case of the ellipsoid, any line through the affine center would be an affine axis of rotation, and we have just seen that a generic (i.e. non-revolution) ellipsoid has no affine axes of rotation. 

Next we are going to explore some properties involving affine spheres and affine rotation surfaces. First, we need the following lemma, which follows easily from differentiating the affine co-normal vector (see Eq. \eqref{nu}) with respect to $u,v$, and performing some straightforward calculations. 

\begin{lemma} \label{formula-previous}
Let $\nu$ represent the affine co-normal vector, ${\bf N}$ a unitary normal vector, and $K\not \equiv 0$ the Gaussian curvature  of a surface $S$. Then 
\begin{equation}\label{long}
\left(\nu_u\times \nu_v\right)\cdot {\bf N}=|K|^{-1/2}\cdot \left({\bf N}_u\times {\bf N}_v\right)\cdot {\bf N}.
\end{equation}
\end{lemma}

Lemma \ref{formula-previous} allows us to prove the following result. 

\begin{lemma} \label{atmost}
Let $S$ be a surface whose Gaussian curvature is not identically zero. If all the affine normal lines to $S$ intersect two lines ${\mathcal A}_1$ and ${\mathcal A}_2$, then ${\mathcal A}_1$ and ${\mathcal A}_2$ intersect at a point $P$, and all the affine normal lines go through $P$, i.e. $S$ is an affine sphere centered at $P$. 
\end{lemma}

\begin{proof} If ${\mathcal A}_1\cap {\mathcal A}_2=\{P\}$ then either all the affine normal lines pass through $P$, which satisfies the statement, or all these normals lie on a common plane, namely the plane containing ${\mathcal A}_1$ and ${\mathcal A}_2$. However in this last case, since $S$ is contained in the union of its affine normal lines, we conclude that $S$ is a plane, which cannot happen by hypothesis. This same argument shows that ${\mathcal A}_1$ and ${\mathcal A}_2$ cannot be parallel. So suppose that ${\mathcal A}_1$ and ${\mathcal A}_2$ are skew. Since all the affine normal lines intersect ${\mathcal A}_1$ and ${\mathcal A}_2$, the set of all the affine normal lines of $S$  must span a ruled surface, which contains $S$. Hence $S$ is itself ruled, and its rulings are the affine normal lines of $S$. Now given a point $p\in P$, the ruling through $p$ is perpendicular to the normal ${\bf N}$ to the surface $S$ at $p$. However the direction of the ruling is the direction of the affine normal, so it follows from Equation \eqref{affine-normal} that $\left(\nu_u\times \nu_v\right)\cdot {\bf N}=0$ at each point $p\in S$. By the identity in Eq. \eqref{long} from Lemma \ref{formula-previous}, we conclude that $\left({\bf N}_u\times {\bf N}_v\right)\cdot {\bf N}=0$. However, this equalility implies that the determinant of the second fundamental form of $S$ is identically zero (see Section 2-8 of \cite{Struik}). But then the Gaussian curvature of $S$ is identically zero, which is excluded by hypothesis. 
\end{proof} 

Lemma \ref{atmost}, together with Theorem \ref{th-imp}, implies the following result, which shows the connection between the existence of several axes of rotation, and the notion of an affine sphere. For the third statement of this theorem, we use the fact (see \cite{Loftin-2}) that every affine sphere of parabolic or elliptic type is a quadric. 

\begin{theorem}\label{isaffinesphere}
Let $S$ be an affine rotation surface whose Gaussian curvature is not identically zero. If $S$ has more than one axis of affine rotation, then:
\begin{itemize}
\item [(1)] All the axes of affine rotation of $S$ intersect at a single point $P$. 
\item [(2)] The surface $S$ is an affine sphere centered at $P$. 
\item [(3)] If $S$ is not a quadric, then $S$ is an affine sphere of hyperbolic type. 
\end{itemize}
\end{theorem}

The converse of Theorem \ref{isaffinesphere} is not necessarily true, because not every affine sphere is an affine rotation surface. Additionally, Theorem \ref{isaffinesphere} does not require that $S$ is an affine rotation surface of the same type about several different axes: Theorem \ref{isaffinesphere} also applies when $S$ is an affine rotation surface of different type (say, elliptic and hyperbolic) about two different axes. At the end of Section \ref{prelim-new} we already mentioned that there are affine rotation surfaces of hyperbolic and elliptic type about different axes, which therefore are affine spheres. There are also hyperbolic affine rotation surfaces about infinitely many axes (see Section 3 of \cite{AGHM16}), and elliptic affine rotation surfaces about infinitely many axes (Euclidean spheres). 

Next we want to explore if we can have multiple axes of affine rotation for parabolic affine rotation surfaces. In order to investigate this issue, we need to learn more about parabolic affine rotation surfaces. We study these surfaces in the next subsection.

\subsection{Parabolic affine rotation surfaces.}\label{subsec-parab}

Let us focus now on parabolic affine rotation surfaces, so assume that $F(x,y,z)=0$ defines a parabolic algebraic affine rotation surface about an axis ${\mathcal A}$. If $F(x,y,z)$ has degree $N$, then $F(x,y,z)$ can be written as 
\begin{equation}\label{decompos}
F(x,y,z)=F_N(x,y,z)+\cdots +F_i(x,y,z)+\cdots+F_0(x,y,z),
\end{equation}
where for $i=0,1,\ldots,N$, $F_i(x,y,z)$ is a homogeneous polynomial of degree $i$. The following results for parabolic affine rotation surfaces are analogous to the results for elliptic rotation surfaces provided by Lemma 1, Lemma 2 and Corollary 3 in Subsection 2.1 of \cite{AG14} and the results for hyperbolic rotation surfaces given by Lemma 3, Lemma 4 and Corollary 5 in Subsection 3.1 of \cite{AGHM16}. The proofs are analogous and are left to the reader. 

\begin{lemma} \label{forms-1}
$F_N(\bfx)$ defines a parabolic affine rotation surface about an axis ${\mathcal A}'$ through the origin parallel to ${\mathcal A}$.
\end{lemma}

\begin{lemma} \label{forms-2}
If the axis ${\mathcal A}$ passes through the origin, then for $i=1,2,\ldots,N$, $F_i(\bfx)$ defines a parabolic affine rotation surface about ${\mathcal A}$.
\end{lemma}

\begin{corollary} \label{cor-forms}
$F(\bfx)$ defines a parabolic affine rotation surface about an axis passing through the origin if and only if for $i=1,2,\ldots,N$, $F_i(\bfx)$ defines a parabolic affine rotation surface about the same axis passing through the origin.
\end{corollary}

Furthermore, recall from Section \ref{prelim-new} that if $S$ is a parabolic affine rotation surface about the axis ${\mathcal A}$, the parallel curves of $S$ are parabolas whose axes of symmetry are, all of them, parallel to ${\mathcal A}$. Furthermore, all these parabolas lie on planes normal to a fixed direction, that we will denote by $L_{\mathcal A}$, perpendicular to ${\mathcal A}$. We will refer to the direction of $L_{\mathcal A}$ as the \emph{normal direction}. Now we have the following theorem, similar to Theorem 6 in \cite{AG14} and Theorem 6 in \cite{AGHM16}.

\begin{theorem}\label{th-general}
Let $F(\bfx)$ define a parabolic affine rotation surface with axis ${\mathcal A}$, and let $F_N(\bfx)$ be the form of highest degree of $F(\bfx)$. Then:
\begin{itemize}
\item [(1)] $F_N(\bfx)$ factors completely into a single real plane $\Pi$ and several quadrics that can be either real or imaginary, possibly with multiplicity in both cases. The real plane and the quadrics may or may not be present in the factorization.
\item [(2)] The real quadrics are elliptic or circular cones, with vertex at the origin. The elliptic cones have one axis of symmetry normal to both ${\mathcal A}$ and $L_{\mathcal A}$. The axis of revolution of the circular cones forms an angle of $\frac{\pi}{4}$ with ${\mathcal A}$. 
\item [(3)] For the non-real quadrics, if any, the associated matrix has one real eigenvalue, whose associated eigenvector is normal to both ${\mathcal A}$ and $L_{\mathcal A}$.
\item [(4)] The direction of the normal vector to the plane $\Pi$ is normal to ${\mathcal A}$. The planes parallel to $\Pi$ provide the family of parallel curves of $F(\bfx)=0$, which are parabolas whose axes of symmetry are parallel to ${\mathcal A}$. 
    \end{itemize}
\end{theorem}

\begin{proof} By Lemma \ref{forms-1}, $F_N(\bfx)$ defines a parabolic affine rotation surface about a line ${\mathcal A}'$ through the origin, parallel to ${\mathcal A}$. Since the nature of the factors of 𝐹$F_N(\bfx)$ is not affected by a rigid motion, we can begin by assuming that ${\mathcal A}'$ is the $z$-axis, and that $F_N(\bfx)$ is invariant under the matrix group in Eq. \eqref{final}. Now recall from Section \ref{prelim-new} that $F_N(\bfx)=\tilde{F}(y^2-2xz,x)$. Hence since $F_N(\bfx)$ is algebraic, $F_N(\bfx)$ can be written as
\[
\begin{array}{l}
F_N(\bfx)=a_0(y^2-2xz)^p x^q+a_1(y^2-2xz)^{p-1}x^{q+2}+\cdots\\
\hspace{1.3 cm}+a_j(y^2-2xz)^{p-j}x^{q+2j},
\end{array}
\]for some $j\in\{0,1,\ldots,p\}$ where $2p+q=N$. Therefore we deduce that
\begin{equation}\label{eq-fac}
\begin{array}{l}
F_N(\bfx)=(y^2-2xz)^{p-j}\cdot x^q\cdot [a_j x^{2j}+\cdots\\
\hspace{1.3 cm}+a_1(y^2-2xz)^{j-1}x^2+a_0 (y^2-2xz)^j ].
\end{array}
\end{equation}
The first factor, of multiplicity $p-j$, corresponds to a real cone of revolution with vertex at the origin, and whose axis is the line $\{x=z,y=0\}$ (see Appendix I). The second factor, of multiplicity $q$, is the $yz$-plane, normal to the $x$-axis.
As for the last polynomial, we shall seek factors for this polynomial of the form $y^2-2xz-wx^2$. A thorough analysis of the surface implicitly defined by such a quadratic form is carried out in Appendix I. In particular, if $w$ is real, $y^2-2xz-wx^2=0$ corresponds to a real, elliptic or circular cone with vertex at the origin, with a symmetry axis normal to the $z$-axis, see Appendix I. To determine the values of $w$ for which $y^2-2xz-wx^2$ is a factor of \eqref{eq-fac} we substitute $$y^2-2xz=wx^2$$ into the last polynomial in \eqref{eq-fac}, and after factoring out $x^{2j}$ we get
\[a_j+\cdots+a_1w^{j-1}+a_0w^j=0.\]This expression is a univariate polynomial in $w$, so this polynomial has $j$ solutions, possibly with multiplicity, possibly complex. Since the nature of the factors is not affected by a rigid motion, and taking into account the analysis done in Appendix I, the theorem follows. 
\end{proof}

\begin{corollary} \label{not-possible}
If $S$ is a parabolic affine rotation surface about an axis ${\mathcal A}$, then $S$ cannot be an affine rotation surface of a different type about the same axis ${\mathcal A}$. 
\end{corollary}

\begin{proof} Without loss of generality, we can assume that the axis of affine rotation of $S$ is the $z$-axis. Then each homogeneous form $F_i(x,y,z)$ of $F(x,y,z)$, after perhaps a rotation about the $z$-axis, factors as shown in Theorem \ref{th-general}. If $S$ is also a surface of revolution about the $z$-axis then the cross sections of $F_i(x,y,z)=0$ with planes $z=c$ are circles, which is impossible. If $S$ is also a hyperbolic affine rotation surface about the $z$-axis then the cross sections of $F_i(x,y,z)=0$ with planes $z=c$ are rectangular hyperbolas centered at the point $(0,0,c)$. However, even if $F_i(x,y,z)$ does not have linear factors, this is not possible, because the curves $\{y^2-2cx-wx^2=0,z=c\}$ are not rectangular hyperbolas: if $w=0$ the curve is a parabola, and if $w\neq 0$ then the curve is a hyperbola centered at the point $\left(\frac{2c}{w},0,c \right)$.
\end{proof}

In fact, excluding planes, which are the only surfaces which may be simultaneously affine rotation surfaces of elliptic and hyperbolic type about the same axis (see Lemma 15 in \cite{AGHM16}), two algebraic surfaces cannot be an affine rotation surface of two different types about the same axis ${\mathcal A}$.

\begin{remark}\label{pi} It can happen that a parabolic affine rotation surface about an axis is simultaneously an affine rotation surface of another type about another axis. For instance, the surface $S$ defined by $y^2-2xz=0$ is a parabolic affine rotation surface about the $z$-axis. But applying the following orthogonal change of coordinates,
\[
x:=\frac{1}{\sqrt{2}}\tilde{x}-\frac{1}{\sqrt{2}}\tilde{z},\mbox{ }y:=\tilde{y},\mbox{ }z:=\frac{1}{\sqrt{2}}\tilde{x}+\frac{1}{\sqrt{2}}\tilde{z},
\]
the surface is mapped to $\tilde{x}^2=\tilde{y}^2+\tilde{z}^2$, where we recognize (see Section 3.2 of \cite{AGHM16}) a hyperbolic affine rotation surface, about \emph{any} line through the origin contained in the plane $\tilde{x}=0$, i.e. \emph{any} line through the origin contained in the plane $x+z=0$. Additionally, $\tilde{x}^2=\tilde{y}^2+\tilde{z}^2$ is also a surface of revolution about the $\tilde{x}$-axis, i.e. an elliptic affine rotation surface about the line $\{y=0,x-z=0\}$. So $S$ is an affine rotation surface of elliptic, parabolic, and hyperbolic type, about three different axes of affine rotation. 
\end{remark}

\begin{remark}\label{how-to-find}
In order to fully understand the geometry of a parabolic affine rotation surface, we need to know not only the axis ${\mathcal A}$, but also the normal direction $L_{\mathcal A}$. Theorem \ref{th-general}, jointly with the results in Appendix I, provides certain clues for how to find the directions of ${\mathcal A}$ and $L_{\mathcal A}$ from the implicit equation $F(x,y,z)=0$, and in some cases, a complete method to do so:
\begin{itemize}
\item If the form $F_N(x,y,z)$ of highest degree has some linear factor, $L_{\mathcal A}$ is normal to the plane $\Pi$ defined by this linear factor. Furthermore, the intersections of $S$ with $\Pi$ are parabolas, and the axes of symmetry of these parabolas provide the direction of ${\mathcal A}$.
\item If the form $F_N(x,y,z)$ has a quadratic factor whose associated matrix $A$ is diagonalizable and has three different eigenvalues, or has two different eigenvalues but is not diagonalizable, $\lambda=1$ must be a simple eigenvalue of the matrix $A$, and the directions of both ${\mathcal A},L_{\mathcal A}$ lie on the plane normal to the eigenvector associated with $\lambda=1$.
\item If the form $F_N(x,y,z)$ has a quadratic factor whose associated matrix $A$ is diagonalizable and has two different eigenvalues, $\lambda=-1$ must be a simple eigenvalue of $A$, and the directions of both ${\mathcal A},L_{\mathcal A}$ lie on the plane normal to the eigenvector associated with $\lambda=-1$.
\end{itemize}
\end{remark}

Next we want to explore whether or not it is possible for a surface to be a parabolic affine rotation surface about two different axes. But first we need some preliminary results. The following lemma is a consequence of Theorem \ref{th-general}. Here, we say that ${\mathcal A}$ is a \emph{parabolic} axis of rotation of $S$, if $S$ is a parabolic affine rotation surface about ${\mathcal A}$; similarly, we can speak about \emph{elliptic} or \emph{hyperbolic} axes of rotation. 

\begin{lemma}\label{axis-z}
Let $F(x,y,z)$ define a parabolic affine rotation surface $S$ whose Gaussian curvature is not identically zero, with $F(x,y,z)$ as in Eq. \eqref{decompos}, invariant under the right-most matrix group in Eq. \eqref{types} (in particular, notice that the $z$-axis is a parabolic axis of rotation of $S$). If $S$ has another parabolic axis of rotation ${\mathcal A}'$, different from the $z$-axis, then: (1) ${\mathcal A}'$ intersects the $z$-axis at the origin; (2) the homogeneous forms $F_i(x,y,z)$ of $F(x,y,z)$ have no linear factors. 
\end{lemma}

\begin{proof} Since by hypothesis $S$ is not a plane, there must exist some non-constant form $F_j(x,y,z)$ not depending only on $x$. Since $F_j(x,y,z)$ is homogeneous, $F_j(x,y,z)$ defines a conical surface with vertex at the origin. Thus, every linear transformation, and in particular every parabolic affine rotation, must preserve the origin, so ${\mathcal A}'$ must also contain the origin, and (1) follows.

Now let us see (2). From the proof of Theorem \ref{th-general}, the only possible linear factor of $F_i(x,y,z)$ is $x$. Suppose that some $F_i(x,y,z)$ has $x$ as a factor. Then ${\mathcal A}'$ must lie on the plane $x=0$, otherwise the intersections of $F_i(x,y,z)=0$ with the planes normal to $L_{{\mathcal A}'}$ would contain a line (the intersection with $x=0$), and therefore those intersections would not be products of parabolas. Thus, $L_{{\mathcal A}'}=L_{\mathcal A}$. However, if ${\mathcal A}'$ is a parabolic axis of rotation of $F_i(x,y,z)=0$, the intersection of $F_i(x,y,z)=0$ with a generic plane $x=c$ must yield a product of parabolas whose common axis of symmetry is parallel to ${\mathcal A}'$. But since the $z$-axis is a parabolic axis of rotation of $F_i(x,y,z)=0$, the intersections of $F_i(x,y,z)=0$ with $x=c$ is a product of parabolas whose axis of symmetry is parallel to the $z$-axis. Therefore since ${\mathcal A}'$ passes through the origin, we conclude that ${\mathcal A}'={\mathcal A}$.
\end{proof}

We also need the following two lemmas. 

\begin{lemma} \label{no-two}
Let $S_1,S_2$ be defined by $y^2-2xz-w_1x^2=0$, $y^2-2xz-w_2x^2=0$, where $w_1\neq w_2$, $w_1$, $w_2$ real or complex. If $\Pi$ is a plane such that $S_1\cap \Pi$ and $S_2\cap \Pi$ are both parabolas, then $\Pi$ is perpendicular to the $x$-axis. 
\end{lemma}

\begin{proof} A plane $\Pi$ not perpendicular to the $x$-axis must be of the form either $\alpha x+y +\beta z+\gamma=0$, or $\alpha x+z+\gamma=0$. Let us consider the first case; the analysis for the second case is similar. Now let $S$ be the surface defined by $y^2-2xz-wx^2=0$. Since $S$ is a quadric, $S\cap \Pi$ must be a conic. Furthermore, if $S\cap \Pi$ is a parabola, then $S\cap\Pi$ has just one point at infinity. In order to find the points at infinity of $S\cap \Pi$, we introduce a homogenizing variable $u$. Then $\alpha x+\alpha y +\beta z+\gamma u=0$ defines the projective closure of $\Pi$. The points at infinity of $S\cap \Pi$ are the points $[x:y:z:0]$, where 
\[
y^2-2xz-wx^2=0,\mbox{ }\alpha x+y+\beta z=0.
\]
Substituting $y=-\alpha x-\beta z$ into the first equation, we deduce that
\begin{equation}\label{double}
\beta^2 z^2+2(\alpha \beta-1)xz+(\alpha^2-w)x^2=0.
\end{equation}
If $\beta=0$, we get two different points at infinity, namely $[0:0:1:0]$ and $[1:-\alpha:\frac{1}{2}(\alpha^2-w):0]$, in which case $S\cap \Pi$ cannot be a parabola. So $\beta\neq 0$. In this case, in order to get just one point at infinity Eq. \eqref{double} must have just one (double) root, which imposes the condition 
\begin{equation}\label{condit}
w=\frac{2\alpha \beta-1}{\beta^2}.
\end{equation}
Therefore, given $y^2-2xz-w_1x^2=0$, $y^2-2xz-w_2x^2=0$, any plane $\alpha x+y+\beta z+\gamma=0$ intersecting both of these surfaces in a parabola must satisfy Eq. \eqref{condit} for both $w_1,w_2$. But this implies $w_1=w_2$, which contradicts the hypothesis $w_1\neq w_2$. A similar line of reasoning shows that the planes $\alpha x+z+\gamma=0$ cannot intersect both $S_1$ and $S_2$ in parabolas. 
\end{proof}

\begin{lemma} \label{new-inf-axes}
Let $S$ be the surface defined by $y^2-2xz-wx^2=0$, where $w\in {\Bbb R}$. 
\begin{itemize}
\item [(1)] If $w=0$, then $S$ has infinitely many parabolic affine rotation axes, namely all the rulings of the cone $y^2-2xz=0$, intersecting at a point (the origin).
\item [(2)] If $w\neq 0$, then $S$ has two distinct parabolic affine rotation axes contained in the $xz$-plane, which are rulings of the cone $y^2-2xz-wx^2=0$ intersecting at a point (the origin). 
\end{itemize}
\end{lemma}

\begin{proof} Since $F(x,y,z)=y^2-2xz-wx^2=G(y^2-2xz,x)$, it is clear that $S$ is a parabolic affine rotation surface about the $z$-axis. Observe that (see Appendix I) $S$ is a cone, an elliptic cone if $w\neq 0$, and a circular cone if $w=0$, whose vertex lies at the origin, and that the $z$-axis $\equiv \{x=0,y=0\}$ is a ruling of $S$. Thus if $w=0$, $S$ is a surface of revolution. Therefore, $S$ is invariant under all the Euclidean rotations about its axis of revolution, which we denote by $L_{\mathcal R}$. Thus, given any ruling ${\mathcal L}\subset S$, there exists an orthogonal transformation (namely, a rotation about $L_{\mathcal R}$) mapping ${\mathcal L}$ onto the $z$-axis, and leaving $S$ invariant. So $S$ is also a parabolic affine rotation surface about any ruling ${\mathcal L}$ of $S$, and (1) follows. 

Now let us see (2). Since $F(x,y,z)=y^2-2xz-wx^2$ satisfies $F(x,-y,z)=F(x,y,z)$, $S$ is symmetric with respect to the plane $y=0$. If $w\neq 0$, the intersection of $S$ with $y=0$ consists of two distinct lines, the $z$-axis and the line ${\mathcal L}\equiv \{y=0,2z+wx=0\}$. Since the surface $S$ is an elliptic cone, there is a line ${\mathcal L}_{\mathcal R}$ such that any plane perpendicular to ${\mathcal L}_{\mathcal R}$ (and not passing through the vertex of the cone) intersects the surface $S$ in an ellipse whose center lies on ${\mathcal L}_{\mathcal R}$. Moreover, these ellipses have parallel, coplanar major and minor axes. Furthermore, since the cone is symmetric with respect to the plane $y=0$, this plane must pass through either all the major axes or all the minor axes of these ellipses. Therefore the line ${\mathcal L}_{\mathcal R}$ must lie in the plane $y=0$. Again by symmetry, the line ${\mathcal L}_{\mathcal R}$ is the angle bisector in the plane $y=0$ of the $z$-axis and the line ${\mathcal L}$, two rulings that pass through the end points of the major or minor axes of the cross sectional ellipses. Hence we can rotate the line ${\mathcal L}$ about the line ${\mathcal L}_{\mathcal R}$ by 180 degrees into the $z$-axis leaving the surface $S$ unchanged. Thus, we deduce that ${\mathcal L}$ is also a parabolic affine rotation axis of $S$.

Finally, let us see now that the $z$-axis and ${\mathcal L}$ are the only parabolic affine rotation axes of $S$. Indeed, by Remark \ref{how-to-find} and Appendix I, any other parabolic axis ${\mathcal L}'$ must lie in a plane normal to the $y$-axis, which is the direction of the eigenvector associated with the eigenvalue $\lambda=1$ of the matrix defining $S$. Since the $z$-axis and ${\mathcal L}$ intersect at the origin, ${\mathcal L}'$ must contain the origin too, so ${\mathcal L}'$ is a line through the origin lying in the plane $y=0$. However, since $S$ is an elliptic cone, if ${\mathcal L}'$ is different from ${\mathcal L}$ and from the $z$-axis, then ${\mathcal L}'$ does not lie on the cone
so the intersection of $S$ with any plane $\Pi'$ parallel to ${\mathcal L}'$ is either an ellipse or a hyperbola. Thus, ${\mathcal L}'$ cannot be a parabolic affine rotation axis. 
\end{proof} 

Lemma \ref{new-inf-axes} leads to the following corollary.

\begin{corollary}\label{cor-inf-axes}
Let $S$ be defined by $y^2-2xz-wx^2=\delta$, where $w,\delta\in {\Bbb R}$. If $w=0$, then $S$ has infinitely many parabolic axes of rotation. If $w\neq 0$, then $S$ has two distinct parabolic axes of rotation.
\end{corollary}

Observe that the surfaces in Corollary \ref{cor-inf-axes} are level surfaces of $F(x,y,z)=y^2-2xz-wz^2$. These surfaces are hyperboloids when $\delta\neq 0$, and elliptical cones (see Appendix I) when $\delta=0$. In particular, the surfaces $y^2-2xz=\delta$ play the role of parabolic affine spheres. When $\delta=0$ these affine spheres are right circular cones, whose rulings are the parabolic axes of rotation. Now, finally, we can prove the following general theorem. 

\begin{theorem}\label{th-parab} Let $S$ be an irreducible, algebraic, real, parabolic affine rotation surface. Then $S$ has either one, or two, or infinitely many parabolic axes of rotation. Furthermore, if $S$ has more than one parabolic axis of rotation, then there exists a similarity mapping $S$ onto a level surface of $F(x,y,z)=y^2-2xz+wx^2$, with $w\in {\Bbb R}$, where $w\neq 0$ if $S$ has two parabolic axes of rotation, and $w=0$ if $S$ has infinitely many parabolic axes of rotation.
\end{theorem}

\begin{proof} Suppose that $S$ has two different parabolic axes of rotation ${\mathcal A}_1,{\mathcal A}_2$, and assume without loss of generality that ${\mathcal A}_1$ is the $z$-axis.   By statement (1) of Lemma \ref{axis-z}, ${\mathcal A}_2$ intersects ${\mathcal A}_1$ at the origin. Furthermore, by Lemma \ref{forms-2} each homogeneous form $F_i(x,y,z)$ of $F(x,y,z)$ must be a parabolic affine rotation surface about the axes ${\mathcal A}_1,{\mathcal A}_2$. By statement (2) of Lemma \ref{axis-z}, the $F_i(x,y,z)$ have no linear factors; therefore, each $F_i(x,y,z)$ factors into products of polynomials $y^2-2xz-w_j x^2$, where $w_j$ is possibly complex. Since ${\mathcal A}_2$ is a parabolic axis of rotation, the family of planes orthogonal to $L_{{\mathcal A}_2}$ must intersect each surface $y^2-2xz-w_j x^2=0$ in a parabola. However, from Lemma \ref{no-two} and Lemma \ref{new-inf-axes}, this is only possible if all the $w_j$ are the same. Indeed, we have two possibilities: (A) one of the $w_j$ is 0; (B) both $w_j$ are nonzero. In case (A), by Lemma \ref{new-inf-axes} the factor $y^2-2xz$ represents a parabolic affine rotation surface with infinitely many axes, namely all the rulings of the cone $y^2-2xz=0$. Also by Lemma \ref{new-inf-axes} the factor $y^2-2xz-w_jx^2$ with $w_j\neq 0$ represents a parabolic affine rotation surface with two different axes, ${\mathcal A}_1$ and ${\mathcal A}_2$, where ${\mathcal A}_2\neq {\mathcal A}_1$ lies on the $xz$-plane. Additionally, ${\mathcal A}_2$ is also a parabolic axis of $y^2-2xz=0$ because ${\mathcal A}_2$ is a ruling of $y^2-2xz=0$. Furthermore, since $w_j\neq 0$ it follows from Lemma \ref{no-two} that $L_{{\mathcal A}_2}$ must be the $x$-axis. Since ${\mathcal A}_2$ contains the origin, ${\mathcal A}_2$ lies in the $xz$-plane and $L_{{\mathcal A}_2}$ is perpendicular to ${\mathcal A}_2$, we deduce that ${\mathcal A}_2={\mathcal A}_1$, which cannot be. A similar argument shows that case (B) cannot happen either.

Therefore, for each $i$, the homogeneous form $F_i$ is either constant $(F_0)$, or has just one irreducible factor, which we represent by $y^2-2xz-wx^2$. Then $F(x,y,z)$ is a polynomial in $y^2-2xz-wx^2$, so $S$ is the union of surfaces of the form $y^2-2xz-wx^2=\delta$, where $w,\delta$ are possibly complex. Since $F(x,y,z)=0$ is irreducible by hypothesis, $S$ has the form $y^2-2xz-wx^2=\delta$. Since $S$ is real, $w,\delta$ are real and the result follows from Corollary \ref{cor-inf-axes}. 
\end{proof}

Therefore, we conclude that there exist parabolic affine rotation surfaces that are affine spheres. From Theorem \ref{isaffinesphere}, examples are the surfaces, characterized in Theorem \ref{th-parab}, with either two or infinitely many parabolic axes, and also the surfaces that are parabolic affine rotation surfaces about an axis and, simultaneously, affine rotation surfaces of another type about another axis.  Figure \ref{fig-infinite} shows the surface $y^2-2xz=0$, a surface with infinitely many parabolic axes of rotation. Two of these axes are the $z$-axis (in blue) and the $x$-axis (in red). Some parallel curves with respect to each axis are also shown. 

\begin{figure}
\begin{center}
$$\begin{array}{c}
\includegraphics[scale=1]{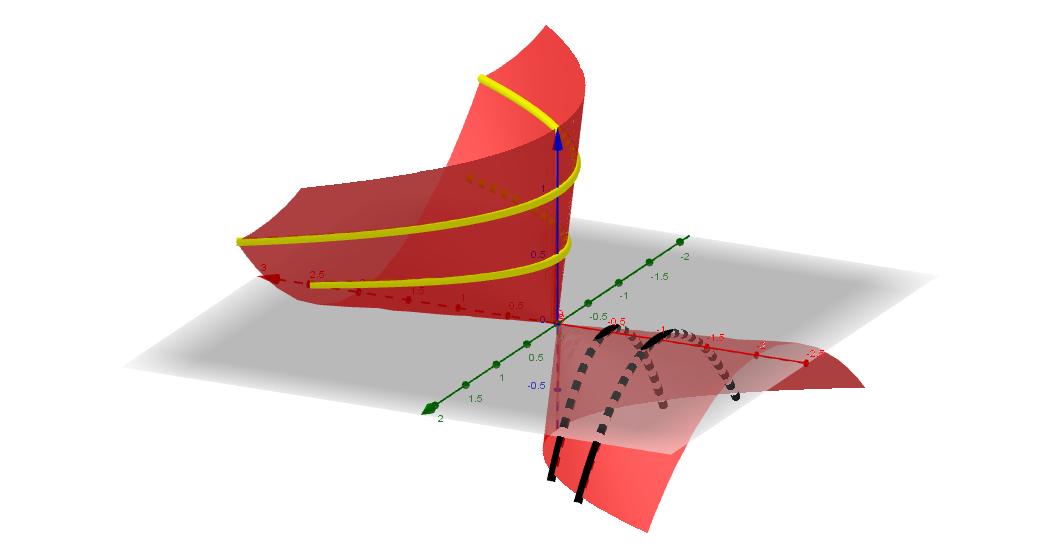}  
\end{array}$$
\end{center}
\caption{The surface $y^2-2xz=0$, a parabolic affine rotation surface with infinitely many affine axes of rotation. Some cross sections parallel to two of these axes (the $z$-axis and the $x$-axis) are shown in black and yellow.}\label{fig-infinite}
\end{figure}

\section{Algebraic affine rotation surfaces.}\label{sec-algebraic}

Let $S$ be an algebraic surface, with the hypotheses listed at the end of Section 2. Our goal here is to provide a method to detect whether or not $S$ is an affine rotation surface by examining the equations of the surface. In order to do so, we begin by recalling \emph{Pl\"ucker coordinates} \cite{Pottmann}. 

Pl\"ucker coordinates provide an alternative way to represent straight lines. A line $L\subset {\Bbb R}^3$ is completely determined when we know a point $P\in L$ and a vector $\bfw$ parallel to $L$. Therefore we often write $L=(P,\bfw)$. Now let $\bar{\bfw}=\bfP\times \bfw$, where $\bfP$ here denotes the vector connecting the point $P$ with the origin of the coordinate system. Then the \emph{Pl\"ucker coordinates} of $L$ are the coordinates of $(\bar{\bfw},\bfw)\in {\Bbb R}^6$. Notice that by construction $\bar{\bfw}\cdot \bfw=0$. 

Pl\"ucker coordinates are unique up to multiplication by a constant. Moreover $\bar{\bfw}$ is independent of the choice of the point $P\in L$, since if $Q\in L$, then $(\bfQ-\bfP)\times \bfw=0$. Furthermore, given the Pl\"ucker coordinates $(\bar{\bfw},\bfw)$ of $L$, we can recover a point $P$ on $L$ from the relationship
\begin{equation}\label{pluck}
\bfP\times \bfw=\bar{\bfw},
\end{equation}
by writing $\bfP=(x,y,z)$ and solving the linear system \eqref{pluck} for $x,y,z$. 

Let $(\beta,\alpha)$ be the Pl\"ucker coordinates for a line in ${\Bbb R}^3$, and consider all the lines $(\bar{\bfw},\bfw)$, written in Pl\"ucker coordinates, such that 
\begin{equation} \label{lin}
\alpha\cdot\bar{\bfw}+\beta\cdot \bfw=0.
\end{equation}
This equation (see \cite{Pottmann}, \cite{Vrseck}) expresses the condition that the line $(\beta,\alpha)$ intersects the line $(\bar{\bfw},\bfw)$. Equation \eqref{lin} corresponds to a hyperplane of ${\Bbb R}^6$, which in \cite{Pottmann} is called a {\it linear complex}. 

Now we will distinguish two cases: when $S$ is given implicitly by an algebraic equation, or when $S$ is given explicitly by a rational parametrization. In the rest of this section, $F(x,y,z)=0$ denotes the implicit equation of $S$. For simplicity, we assume that $F$ is irreducible. Furthermore, $\bfx(t,s)$ represents a rational parametrization of $S$.   

\subsection{The implicit case.} \label{subsec-implicit}

Let $\xi(x,y,z)$ be the affine normal vector at a generic, non-singular point $(x,y,z)$ of the surface $S$ implicitly defined by $F(x,y,z)=0$. An explicit expression for the affine normal vector $\xi(x,y,z)$ can be found in Appendix A of \cite{Andrade}. One can check that 
\begin{equation}\label{xi-gamma}
\xi(x,y,z)=a(x,y,z) (\gamma_1(x,y,z),\gamma_2(x,y,z),\gamma_3(x,y,z)),
\end{equation}
where $a(x,y,z)$ involves a radical term (see the last line on page 75 of \cite{Andrade}), and for $i=1,2,3$, $\gamma_i(x,y,z)$ is a rational function of $x,y,z$; the expressions for the $\gamma_i$ are very long and can be found on pages 75-77 of \cite{Andrade}. Calling $\gamma=(\gamma_1,\gamma_2,\gamma_3)$ and $\overline{\gamma}=\bfx\times \gamma$, with $\bfx=[x,y,z]^T$, we can take $(\overline{\gamma},\gamma)$ as Pl\"ucker coordinates for the affine normal line to $S$ at a generic non-singular point $(x,y,z)$. Notice that the components of $(\overline{\gamma},\gamma)$ are rational functions of $x,y,z$. 

The following proposition establishes that the invariance of the surface $S$ is equivalent to the invariance of the polynomial $F(x,y,z)$ implicitly defining $S$. 

\begin{proposition}\label{const-one} Consider a surface $S$ implicitly defined by $F(x,y,z)=0$, with $F$ irreducible. Then $S$ is invariant under a group of affine rotations about the line ${\mathcal A}$, denoted as ${\mathcal G}_{\alpha,{\mathcal A}}(\bfx)$, iff 
\begin{equation}\label{eq-fundam}
F({\mathcal G}_{\alpha,{\mathcal A}}(\bfx))=F(\bfx).
\end{equation}
\end{proposition}

\begin{proof} ``$\Leftarrow$": If Eq. \eqref{eq-fundam} holds, $F(\bfx)$ and $F({\mathcal G}_{\alpha,{\mathcal A}}(\bfx))$ both define the surface $S$. Therefore, $S$ is invariant under the group ${\mathcal G}_{\alpha,{\mathcal A}}(\bfx)$. ``$\Rightarrow$" We outline the proof for parabolic affine rotation surfaces; the proof is similar for the elliptic and the hyperbolic types. Applying if necessary an orthogonal transformation, we can assume that ${\mathcal A}$ is the $z$-axis. By Corollary \ref{cor-forms} and the proof of Theorem \ref{th-general}, every homogeneous form $F_i(x,y,z)$ of $F(x,y,z)$ is the sum of terms of the form $(y^2-2xz)^r x^{i-2r}$. But both $y^2-2xz$ and $x$ are invariant under ${\mathcal G}_{\alpha,{\mathcal A}}(\bfx)$. Therefore, $F_i\circ  {\mathcal G}_{\alpha,{\mathcal A}}=F_i$ for all $i$, so $F\circ {\mathcal G}_{\alpha,{\mathcal A}}=F$. 
\end{proof} 

The following observation, analogous to Theorem 2.4 in \cite{Vrseck}, is crucial. Here and in subsequent results, we denote the surface defined by $F(x,y,z)=k$, where $k\in {\Bbb R}$, by $S_k$.

\begin{theorem} \label{crucial}
The surface $F(x,y,z)=0$ is an affine rotation surface about an axis ${\mathcal A}$ if and only if for all $k\in {\Bbb R}$, the surface $S_k$ is also an affine rotation surface about the same axis ${\mathcal A}$.  
\end{theorem}

\begin{proof} The implication $(\Leftarrow)$ is straightforward, so we focus on $(\Rightarrow)$. Let ${\mathcal A}$ be the affine rotation axis. Now $F^{(k)}(x,y,z)=F(x,y,z)-k$ represents an affine rotation surface about the line ${\mathcal A}$ iff there exists a rigid motion $T(\bfx)=Q\bfx+b$, independent of $k$, mapping ${\mathcal A}$ to the $z$-axis, such that $\tilde{F}^{(k)}=F^{(k)}\circ T$ is an affine rotation surface about the $z$-axis. We claim that $T$ is any rigid motion mapping ${\mathcal A}$ to the $z$-axis, such that in the new system of coordinates, the surface defined by $F$ is invariant under some uniparametric canonical matrix group $\bfQ_{\alpha}$. Such a transformation $T$ exists because by hypothesis the surface defined by $F$ is invariant under some group $\bfQ_{\alpha}$. So let us see that $\tilde{F}^{(k)}$ is also invariant under the same group $\bfQ_{\alpha}$. Indeed, using Proposition \ref{const-one}, applied on $F$, we get 
\[
\tilde{F}^{(k)}(\bfQ_{\alpha}(\bfx))=(F\circ T)(\bfQ_{\alpha}(\bfx))-k=(F\circ T)(\bfx)-k=\tilde{F}^{(k)}(\bfx).
\]
Now the implication follows from the implication ``$\Leftarrow$" of Proposition \ref{const-one}, applied this time on $\tilde{F}^{(k)}$.
\end{proof}

Therefore we have the following corollary, analogous to Corollary 12 of \cite{AG17}.

\begin{corollary} \label{cor-cruc}
If the surface $F(x,y,z)=0$ is an affine rotation surface then all the affine normals of the surfaces $\{S_k\}_{k\in {\Bbb R}}$ belong to a common hyperplane \eqref{lin} in ${\Bbb R}^6$.
\end{corollary}

Now we can give a necessary condition for a polynomial $F(x,y,z)$ to represent an affine rotation surface. Recall here the definition of $\gamma=\gamma(x,y,z)$ from Eq. \eqref{xi-gamma}. 

\begin{theorem} \label{rev}
Let $F(x,y,z)$ be an irreducible polynomial defining an affine rotation surface $S$ about an axis ${\mathcal A}$. Let $\bfP=(x,y,z)$ be an arbitrary point in ${\Bbb R}^3$, and denote the Pl\"ucker coordinates of ${\mathcal A}$ by $(\beta,\alpha)$, where $\alpha=(\alpha_1,\alpha_2,\alpha_3)$ and $\beta=(\beta_1,\beta_2,\beta_3)$. Then  
\begin{equation} \label{eq3}
\begin{array}{l}
\alpha\cdot (\bfP\times \gamma)+\beta\cdot \gamma=\\
\alpha_1 (y\gamma_3-z\gamma_2)+\alpha_2 (-x\gamma_3+z\gamma_1) +\alpha_3 (x\gamma_2-y\gamma_1)
+
\beta_1 \gamma_1+\beta_2 \gamma_2+\beta_3 \gamma_3\equiv 0.
\end{array}
\end{equation}
\end{theorem}

\begin{proof} If $F(x,y,z)=0$ defines an affine rotation surface about ${\mathcal A}=(\beta,\alpha)$, then by Theorem \ref{th-imp} the set of affine normal lines to $S$ belongs to the linear complex \eqref{lin}. In addition, since the Pl\"ucker coordinates of a generic affine normal line are given by $(\bfP\times \gamma,\gamma)$, the polynomial on the left hand-side of \eqref{eq3} vanishes at every point of $S$. Moreover, by Corollary \ref{cor-cruc} the polynomial on the left hand-side of \eqref{eq3} vanishes at every point of $S_k$ for every value of $k$. Hence, the polynomial in Equation \ref{eq3} vanishes everywhere, since each point $P=(x,y,z)\in {\Bbb R}^3$ lies on $S_k$ for some value of $k$.
\end{proof} 

Notice that Eq. \eqref{eq3} is an identity between two polynomials in $x,y,z$ (the polynomial on the right hand-side of Eq. \eqref{eq3} is the zero polynomial). In particular, $(x,y,z)$ does {\it not} need to be a point on the surface $S$. Hence, in practice, we can derive a linear system of homogeneous equations for the $\alpha_i$ and $\beta_j$ by picking five points at random, not necessarily on $S$ (so that the points have rational coordinates), and substituting these coordinates into Eq. \eqref{eq3}. Certainly, Eq. \eqref{eq3} itself provides such a linear system of equations (the left hand-side, seen as a polynomial in $x,y,z$, must be the zero polynomial), but deriving this system directly is computationally much more expensive than substituting with concrete triplets $(x,y,z)$. Let ${\mathcal L}$ denote the linear system derived by substitution. The solutions of ${\mathcal L}$ are the Pl\"ucker coordinates of the tentative axes of rotation. A first corollary of Theorem \ref{rev} is the following. 

\begin{corollary} \label{cor-rev}
If the coefficients of $F(x,y,z)$ are rational numbers and $S$ is an affine rotation surface about an axis ${\mathcal A}$, then ${\mathcal A}$ admits rational Pl\"ucker coordinates.
\end{corollary}

Let us now see how to take advantage of Theorem \ref{rev} in order to detect affine rotation surfaces. Let $D$ be the dimension of the solution space of ${\mathcal L}$. Then there are three different possibilities: 
\begin{itemize}
\item [(1)] $D=0$: in this case, from Theorem \ref{rev} we conclude that $S$ is not an affine surface of rotation.
\item [(2)] $D=1$: in this case, from Theorem \ref{rev} we conclude that ${\mathcal A}=(\beta,\alpha)$, where $(\beta,\alpha)$ is the solution of ${\mathcal L}$, is the only possible affine axis of rotation of $S$. Notice that the fact that $D=1$ is not enough to ensure that $S$ is an affine rotation surface, since by Theorem \ref{th-imp} we need that $S$ also  
satisfies the shadow line property. However, from \cite{Lee} we know that if $S$ is an affine rotation surface, then $S$ must be of either elliptic, or hyperbolic, or parabolic type.
\item [(3)] $D>1$: in this case, we deduce that there are several lines ${\mathcal A}_i$ that intersect all the affine normal lines of $S$, so according to Lemma \ref{atmost}, all the ${\mathcal A}_i$ intersect at a point $P$, and $S$ is an affine sphere centered at $P$. In this case, $S$ may or may not be an rotation surface, since not every affine sphere is an affine rotation surface.
\end{itemize}

Now we have the following result, which follows from these observations and Remark \ref{rem-con}. Notice that parabolic affine rotation surfaces are not included here. We will address these surfaces separately after our discussion of Theorem \ref{th-final-alg}.

\begin{theorem}\label{th-final-alg}
Suppose that $D=1$. Let ${\mathcal A}$ be the line whose Pl\"ucker coordinates correspond to the solution space of ${\mathcal L}$, and let $\Pi$ be a generic plane normal to ${\mathcal A}$.
\begin{itemize}
\item [(1)] If $S\cap \Pi$ factors into concentric circles, then $S$ is an elliptic affine rotation surface with axis ${\mathcal A}$. 
\item [(2)] If $S\cap \Pi$ factors into rectangular hyperbolas with the same center and asymptotes, then $S$ is a hyperbolic affine rotation surface with axis ${\mathcal A}$. 
\end{itemize}
\end{theorem}

Theorem \ref{th-final-alg} provides a method to detect whether or not $S$ is an affine rotation surface of elliptic or hyperbolic type in the case $D=1$. In practice, in order to check whether or not $S$ is an elliptic or hyperbolic affine rotation surface, after computing ${\mathcal A}$ we can pick a random plane $\Pi$ normal to ${\mathcal A}$, and factor $S\cap \Pi$ to check whether or not $S\cap \Pi$ factors completely into concentric circles, or rectangular hyperbolas with the same center and asymptotes. The process can be made simpler by first applying a rigid motion so that ${\mathcal A}$ is transformed into the $z$-axis.

In order to check if $S$ is a parabolic affine rotation surface we need, as observed in Remark \ref{rem-con}, another strategy. Since ${\mathcal A}$ is known we can find the normal direction $L_{\mathcal A}$ by first implicitizing the surface $S$ and then using the ideas in Remark \ref{how-to-find}. (An absolute factorization of a multivariate polynomial can be computed, for instance, using the command {\tt AFactor} in Maple 18, which works quickly and efficiently). 
Then we can pick a random plane $\Pi'$ normal to $L_{\mathcal A}$, and factor $S\cap \Pi'$ to check whether or not $S\cap \Pi'$ factors completely into parabolas. 

This approach provides a heuristic algorithm to check whether or not $S$ is an affine rotation surface in the case when $S$ is not an affine sphere, and to find the affine rotation group leaving $S$ invariant. A deterministic algorithm can then be built by checking if $S$ is invariant under the corresponding rotation group (elliptic, hyperbolic or parabolic, depending on the nature of the parallel curves) using Proposition \ref{const-one}. 

If $D>1$, in which case $S$ is necessarily an affine sphere, we need a different strategy. In this case $S$ might have different axes of elliptic, hyperbolic, and parabolic types intersecting at a point $P$, the center of $S$ seen as an affine sphere. We can find these axes from the solution of $D$. The elliptic axes, if any, can be computed by using the results in \cite{AG17}. The hyperbolic axes, if any, can be computed by using the results in \cite{AGHM16}. The parabolic axes, if any, can be found by using the results in Section \ref{subsec-parab}. From Remark \ref{how-to-find} in Section \ref{subsec-parab}, if $F(x,y,z)$ has a nonzero linear form, or $F_N(x,y,z)$ has some linear factor, then the direction of the parabolic axis can be determined; since we also have a point on the axis, namely $P$, the axis can be computed. Otherwise (see also Remark \ref{how-to-find}) we have a point $P$ on the axis and tentative normal planes that contain the axis. Next we can pick a generic direction ${\bf v}=(a,b,c)\in {\Bbb R}^3$ in each of these planes, and impose that $S$ is a parabolic affine rotation surface about the axis through $P$, parallel to ${\bf v}$; since we can always choose one of the coordinates of ${\bf v}$ to be 1, the problem reduces to solving bivariate polynomial systems.

\subsection{The rational case.} \label{sec-theratcase}

In this subsection we assume that $S$ is defined by a rational parametrization ${\bf x}(u,v)$. Using Eq. \eqref{nu}, one can check that the affine co-normal vector $\nu=\beta^{-1/4} ({\bf x}_u \times {\bf x}_v)$, where $\beta=\beta(u,v)$ is the product of the determinants of the first and second fundamental forms of ${\bf x}(u,v)$. Since ${\bf x}(u,v)$ is rational, $\beta$ and the components of ${\bf x}_u,{\bf x}_v$ are rational functions of $u,v$. Furthermore, a straightforward computation shows that $\nu_u\times \nu_v=\beta^{-3/2} \Phi$, where $\Phi=\Phi(u,v)$ has rational components. Hence
\[[\nu,\nu_u,\nu_v]=\nu\cdot (\nu_u\times \nu_v)=\beta^{-7/4} ({\bf x}_u \times {\bf x}_v)\cdot \Phi,\]
where $({\bf x}_u \times {\bf x}_v)\cdot \Phi$ is a rational function. Therefore, using Eq. \eqref{affine-normal}, we find that the affine normal vector $\xi=\beta^{1/4} \frac{\Phi}{({\bf x}_u \times {\bf x}_v)\cdot \Phi}$, where $\mu=\frac{\Phi}{({\bf x}_u \times {\bf x}_v)\cdot \Phi}$ has rational components in $u,v$. Calling $\overline{\mu}={\bf x}\times \mu$, where ${\bf x}={\bf x}(u,v)$, we can take $(\overline{\mu},\mu)$ as Pl\"ucker coordinates for the affine normal line to $S$ at a generic non-singular point ${\bf x}(u,v)$. Notice that the components of $(\overline{\mu},\mu)$ are rational functions of $u,v$.  

Now from Theorem \ref{rev} we get the following necessary condition for a rational parametrization ${\bf x}(u,v)$ to represent a generalized affine surface of rotation.

\begin{theorem} \label{rev-rat}
Let ${\bf x}(u,v)$ be rational parametrization. If ${\bf x}(u,v)$ defines a generalized affine rotation surface $S$ about an axis ${\mathcal A}$ with Pl\"ucker coordinates $(\beta,\alpha)$, where $\alpha=(\alpha_1,\alpha_2,\alpha_3)$ and $\beta=(\beta_1,\beta_2,\beta_3)$, then
\begin{equation} \label{eq3-new}
\begin{array}{l}
\alpha\cdot ({\bf x}(u,v)\times \mu(u,v))+\beta\cdot \mu(u,v)= 0.
\end{array}
\end{equation}
\end{theorem}

Eq. \eqref{eq3-new} provides a linear system of equations ${\mathcal L}$ in $\alpha_i,\beta_i$, and the solution provides the Pl\"ucker coordinates of a tentative axis of rotation. As in the implicit case, if $D$ is the dimension of the solution space of ${\mathcal L}$, we need $D\geq 1$ for $S$ to be an affine rotation surface, and we must distinguish the cases $D=1$ and $D> 1$. If $D=1$, in order to check whether or not $S$ is an affine rotation surface of either elliptic or hyperbolic type, we proceed as in the implicit case: we compute the axis ${\mathcal A}$ from the solution space of ${\mathcal L}$, we apply a rigid motion so that ${\mathcal A}$ is transformed into the $z$-axis, and then we analyze the intersection of $S$ with a random plane $\Pi$ normal to the $z$-axis. In the rational case, in general we need to compute the implicit equation of $S\cap \Pi$ in order to analyze the nature of $S\cap \Pi$. If $D>1$, one can still check whether or not $S$ is an elliptic affine rotation surface by applying the results in \cite{AG17}.

However, in order to check whether or not $S$ is a parabolic affine rotation surface, both in the cases $D=1$ and $D>1$, the situation is more complicated, since we need to find the normal direction $L_{\mathcal A}$. Currently, the only method we can provide is to compute the form of highest degree $F_N(x,y,z)$ of the implicit equation, and then apply the ideas in Remark \ref{how-to-find}. A similar situation arises in the case $D>1$ in order to check whether or not $S$ is a hyperbolic affine rotation surface. In this case, the only method that we can suggest at the moment is implicitizing the surface, and then using the method in Section \ref{subsec-implicit}.

\subsection{Algorithms and examples.}

The ideas in the previous sections give rise to two algorithms that can be found in Appendix II, {\tt Affine Rotations-Impl} and {\tt Affine Rotations-Rat}, to check whether or not $S$ is an implicit or rational affine surface of rotation, and if so to find the type and the axis of rotation. The following examples illustrate the main ideas behind these algorithms.

\begin{example} Let $S$ be the quartic surface implicitly defined by 
\[
F(x,y,z)=2xy^3-6xy^2z+6xyz^2-2xz^3+4y^3z-8y^2z^2+4yz^3-y+z-1.
\]
We consider the linear system ${\mathcal L}$ provided by Eq. \eqref{eq3} for several random points. For instance, for the point $x=-59$, $y=6$, $z=5$, we get $485\alpha_2-582\alpha_3-1909\beta_1+204\beta_2+170\beta_3=0$. Similarly, for $x=-\frac{322}{27}$, $y=6$, $z=3$ we get $747\alpha_2-1494\alpha_3-17139\beta_1+8748\beta_2+4374\beta_3=0$. The solution of the system ${\mathcal L}$ is 
\[
\alpha_2=\alpha_3=\beta_1=\beta_2=\beta_3=0,
\]
which corresponds to the $x$-axis ($\alpha_1=1$). In particular, the dimension $D$ of the solution space of ${\mathcal L}$ is $D=1$. In order to see whether or not $S$ is an affine rotation surface, we intersect $S$ with the plane $x=1$. This intersection generates the space curve 
\[
\{4y^3z-8y^2z^2+4yz^3+2y^3-6y^2z+6yz^2-2z^3-y+z-1=0,\mbox{ }x=1\},
\]
which is not the union of circles or hyperbolas (in fact, this polynomial in $y,z$ is irreducible over the complex numbers). Therefore, if $S$ is an affine rotation surface, $S$ must be of parabolic type. Since the form of highest degree of $F(x,y,z)$ is 
\[
F_N(x,y,z)=(y-z)^2(xy-xz+2yz),
\]
if $S$ is a parabolic affine rotation surface, the intersection of $S$ with planes $y-z=c$ must yield a product of parabolas with axis parallel to the $x$-axis. The linear change of coordinates 
\[
\left\{x:=z,\mbox{ }y=\frac{1}{2}(-x+y),\mbox{ }z=\frac{1}{2}(x+y)\right\}
\]
takes the $x$-axis to the $z$-axis, and the planes $y-z=c$ to the planes $x=c'$. Applying this change of coordinates, the surface $S$ is transformed into the surface 
\[
\tilde{F}(x,y,z)=x^2(y^2-2xz-x^2)+x-1,
\]
where we can recognize a parabolic affine rotation surface about the $z$-axis. Therefore, we conclude that the given surface $S$ is a parabolic affine rotation surface, whose axis ${\mathcal A}$ is the $x$-axis.
\end{example}

\begin{example} Let $S$ be the surface parametrized by \[{\bf x}(u,v)=(x(u,v),y(u,v),z(u,v)),\]where
\[
x(u,v)=\frac{-u^3(v^2+1)}{2v},\mbox{ }y(u,v)=\frac{u^3(v^2-1)}{2v},\mbox{ }z(u,v)=u^2+1.
\]
We consider the linear system ${\mathcal L}$ provided by Eq. \eqref{eq3-new} for $u=1,2,3$ and $v=1,2,3$. For instance, for $u=1,v=1$, Eq. \eqref{eq3-new} yields $10\alpha_2+6\beta_1-2\beta_3=0$. Similarly, for $u=1,v=2$, we get $15\alpha_1+25\alpha_2+15\beta_1-9\beta_2-4\beta_3=0$. The solution of the system ${\mathcal L}$ is 
\[
\alpha_1=\alpha_2=\beta_1=\beta_2=\beta_3=0,
\]
which corresponds to the $z$-axis ($\alpha_3=1$). In particular the dimension $D$ of the solution space is $D=1$. In order to see whether or not $S$ is an affine rotation surface, we intersect $S$ with a plane $\Pi$ normal to the $z$-axis. In our case we take $\Pi$ to be the plane $z-10=0$. Then $z(u,v)=u^2+1=10$, so $u^2=9$ and $u=\pm 3$. Therefore, the intersection $S\cap \Pi$ is the union of the curves parametrized by
\[
\left(\frac{-27(v^2+1)}{2v},\frac{27(v^2-1)}{2v},10\right),\mbox{ }\left(\frac{27(v^2+1)}{2v},\frac{-27(v^2-1)}{2v},10\right),
\]
which are two rectangular hyperbolas with the same center and asymptotes. Hence, if $S$ is an affine rotation surface, it must be of hyperbolic type. In fact, in this case it is easy to recognize that $S$ is the hyperbolic affine rotation surface generated by Eq. \eqref{local-param}, where $\bfQ_{\alpha}$ is the matrix in the middle of Eq. \eqref{types}, and with directrix $(u^3,0,u^2+1)$. 
\end{example}

\begin{example} Let $S$ be the ellipsoid $4x^2+16y^2+z^2=1$. In this case, after picking random points the linear system ${\mathcal L}$ provided by Eq. \eqref{eq3} provides the solution 
\[
\alpha_1 = t,\mbox{ }\alpha_2 =s,\mbox{ }\alpha_3 = w,\mbox{ }\beta_1 =\beta_2 =\beta_3 = 0.
\]
In this case, the dimension $D$ of the solution space is $D=3$, so $S$ is an affine sphere with center at $\beta=(0,0,0)$. Since $S$ is not a surface of revolution, $S$ is not an elliptic affine rotation surface. And since no family of planes intersects $S$ in hyperbolas or parabolas, we conclude that $S$ is not an affine rotation surface. 
\end{example}

\section{Conclusion}

We have investigated a generalization of surfaces of revolution, called affine rotation surfaces, well-known in affine differential geometry. We have provided a practical computational algorithm to detect whether or not an algebraic surface, given in either implicit form or in some cases in rational parametric form, is an affine rotation surface, and to identify the type (elliptic, hyperbolic, parabolic) of the surface and the axis of rotation of the surface. These results generalize previous results of the authors for surfaces of revolution, namely Theorem 13 in \cite{AG17}. 

Additionally, we have investigated some properties of algebraic affine rotation surfaces of parabolic type, and several properties of algebraic affine spheres, surfaces whose affine normal lines intersect at a single point. In order to derive these results, we needed the notion of an affine normal, and the well-known characterization of affine rotation surfaces as surfaces whose affine normal lines intersect a common line. Since the affine normal is not defined when the surface has vanishing Gaussian curvature, surfaces with vanishing Gaussian curvature were excluded from our study. Hence, a deeper analysis of surfaces with vanishing Gaussian curvature -- i.e. developable surfaces -- is posed here as a pending open problem for future research. 

Further open questions are suggested by the following issue. To analyze certain cases of rational surfaces, we need to compute the implicit equation of the surface, or at least the highest order form of the implicit equation, which can be costly. This task suggests two open problems: first, finding a better method to detect affine rotation surfaces given in rational parametric form, in all possible cases, that does not require implicitization; second, finding an efficient method to implicitize affine rotation surfaces of parabolic and hyperbolic type. These are questions that we would like to explore in the future. 

\vspace{0.2 cm}
\noindent{\bf Acknowledgements.} Juan G. Alc\'azar is supported by the Spanish Ministerio de Econom\'{\i}a y Competitividad and by the European Regional Development Fund (ERDF), under the project  MTM2017-88796-P, and is a member of the Research Group {\sc asynacs} (Ref. {\sc ccee2011/r34}). We thank Maria Andrade and F. Manhart for some hints and references on affine differential geometry.

\newpage

\section{Appendix I: the cones $y^2-2xz-wx^2=0$.}

The matrix associated with the quadratic form $y^2-2xz-wx^2$ is 
\[
A=\begin{bmatrix}
-w & 0 & -1\\
0 & 1 & 0\\
-1 & 0 & 0
\end{bmatrix}
\]
The eigenvalues of the matrix $A$ are 
\begin{equation}\label{lambda1}
\lambda_1=1,\mbox{ }\lambda_2=-\frac{1}{2}w+\frac{1}{2}\sqrt{w^2+4},\mbox{ }\lambda_3=-\frac{1}{2}w-\frac{1}{2}\sqrt{w^2+4}.
\end{equation}
Observe that $\lambda_2\cdot \lambda_3=-1$, therefore $\lambda_2$ and $\lambda_3$ are always nonzero. Furthermore, if $w\neq 0$ then $\lambda_i\neq 1$ for $i=2,3$; also, $\lambda_2=\lambda_3$ iff $w^2+4=0$, i.e. $w=\pm 2i$, where $i^2=-1$. Now we distinguish the following three cases:

\begin{itemize}
\item [(1)] \fbox{$w\neq 0,\pm 2i$} The matrix $A$ has three distinct, eigenvalues, so $A$ is diagonalizable and $y^2-2xz-wx^2=0$ defines a cone that is not a surface of revolution. Since $\lambda_2\cdot \lambda_3=-1$, $\lambda_2$ and $\lambda_3$ have opposite signs, so $y^2-2xz-wx^2=0$ defines a real cone iff $w$ is real. Furthermore, in that case we get an elliptic cone. The eigenvector associated with $\lambda=1$ is $(0,1,0)$, which defines one of the symmetry axes of the cone. The eigenvectors associated with the other two eigenvalues are parallel to the $xz$-plane.

\item [(2)] \fbox{$w=\pm 2i$} The matrix $A$ has only two distinct eigenvalues, $\lambda_1=1$ (simple) and $\lambda_2=\mp i$ (double), but $A$ is not diagonalizable. The eigenvector associated with $\lambda=1$ is $(0,1,0)$.   

\item [(3)] \fbox{$w=0$} The matrix $A$ has only two distinct eigenvalues, $\lambda_1=1$ (double) and $\lambda_2=-1$ (simple). The matrix $A$ is diagonalizable, so $y^2-2xz=0$ defines a real cone of revolution about the line through the origin in the direction of the eigenvector associated with $\lambda_2=-1$, namely $(1,0,1)$. Notice that the axis of revolution of this cone forms an angle of $\frac{\pi}{4}$ with both the $x$-axis and the $z$-axis. 

\end{itemize}

\section{Appendix II: algorithms.}

\begin{algorithm}[t!]
\begin{algorithmic}[1]
\REQUIRE A real, algebraic, irreducible surface $S$, defined by an implicit equation $F(x,y,z)$, whose Gaussian curvature is not identically zero.
\ENSURE Whether or not $S$ is an affine sphere, and$/$or an affine rotation surface, the type, and the axes of rotation, if any.
\STATE{compute the vector $\gamma(x,y,z)$ (see Equation \eqref{xi-gamma})}
\FOR{$i=1$ to 5}
\STATE{pick a random point $P_i=(x_i,y_i,z_i)$ with rational coordinates}
\STATE{compute the corresponding equation Eq. \eqref{eq3}}
\ENDFOR
\STATE{let ${\mathcal L}$ be the linear system consisting of the 5 previous equations}
\STATE{compute the dimension $D$ of the solution space of ${\mathcal L}$}
\IF{$D=0$}
\STATE{{\bf return} {\tt $S$ is not an affine rotation surface nor an affine sphere}}
\ENDIF
\IF{$D=1$}
\STATE{{\bf return} {\tt $S$ is not an affine sphere}}
\STATE{compute the tentative axis ${\mathcal A}$ from the solution in $D$}
\STATE{pick a random plane $\Pi$, normal to ${\mathcal A}$}
\STATE{check if $S\cap \Pi$ factors into concentric circles or rectangular hyperbolas with the same center and asymptotes}
\STATE{in the negative case, find the direction of $L_{\mathcal A}$ from Remark \ref{how-to-find}, then pick a generic plane $\Pi'$ normal to $L_{\mathcal A}$, and check if $S\cap \Pi'$ factors into parabolas with the same axis}
\STATE{in each case, use Eq. \ref{eq-fundam} to check if the corresponding affine rotation group leaves $S$ invariant}
\STATE{in the affirmative case, {\bf return} {\tt $S$ is an affine rotation surface}, otherwise {\bf return} {\tt $S$ is not an affine rotation surface}}
\ENDIF
\IF{$D>1$}
\STATE{{\bf return} {\tt $S$ is an affine sphere}}
\STATE{compute the elliptic axes of rotation, if any, using the results in \cite{AG17}}
\STATE{compute the hyperbolic axes, if any, using the results in \cite{AGHM16}}
\STATE{compute the parabolic axes, if any, using the results in Section \ref{subsec-parab}}
\STATE{in each case, check if the corresponding affine group leaves $S$ invariant}
\STATE{if we get some positive response {\bf return} {\tt $S$ is an affine rotation surface}, otherwise {\bf return} {\tt $S$ is not an affine rotation surface}}
\ENDIF
\end{algorithmic}
\caption{{\tt Affine Rotations-Impl}}\label{alg-affine-impl}
\end{algorithm} 

\begin{algorithm}[t!]
\begin{algorithmic}[1]
\REQUIRE A real, algebraic, irreducible surface $S$, defined by a rational parametrization ${\bf x}(u,v)$, whose Gaussian curvature is not identically zero.
\ENSURE Whether or not $S$ is an affine sphere, and$/$or an affine surface of rotation, the type, and the axes of rotation, if any.
\STATE{compute the vector $\mu(u,v)$ (see Section \ref{sec-theratcase})}
\FOR{$i=1$ to 5}
\STATE{pick a random rational pair of parameters $(u_i,v_i)$}
\STATE{compute the corresponding equation Eq. \eqref{eq3-new}}
\ENDFOR
\STATE{let ${\mathcal L}$ be the linear system consisting of the 5 previous equations}
\STATE{compute the dimension $D$ of the solution space of ${\mathcal L}$}
\IF{$D=0$}
\STATE{{\bf return} {\tt $S$ is not an affine rotation surface nor an affine sphere}}
\ENDIF
\IF{$D=1$}
\STATE{{\bf return} {\tt $S$ is not an affine sphere}}
\STATE{compute the tentative axis ${\mathcal A}$ from the solution in $D$}
\STATE{pick a random plane $\Pi$, normal to ${\mathcal A}$}
\STATE{find an implicit representation of $S\cap \Pi$.}
\STATE{check if $S\cap \Pi$ factors into concentric circles or rectangular hyperbolas with the same center and asymptotes}
\STATE{in the negative case, find the implicit equation of $S$, and test whether or not $S$ is a parabolic affine rotation surface, using the algorithm {\tt Affine Rotations-Impl}}
\STATE{in each case, use Eq. \ref{eq-fundam} to check if the corresponding affine rotation group leaves $S$ invariant}
\STATE{in the affirmative case, {\bf return} {\tt $S$ is an affine rotation surface}, otherwise {\bf return} {\tt $S$ is not an affine rotation surface}}
\ENDIF
\IF{$D>1$}
\STATE{{\bf return} {\tt $S$ is an affine sphere}}
\STATE{compute the elliptic axes of rotation, if any, using the results in \cite{AG17}}
\STATE{compute the hyperbolic and parabolic axes, if any, by first finding the implicit equation of $S$, and then applying the algorithm {\tt Affine Rotations-Impl}}
\STATE{in each case, check if the corresponding affine group leaves $S$ invariant}
\STATE{if we get some positive response {\bf return} {\tt $S$ is an affine rotation surface}, otherwise {\bf return} {\tt $S$ is not an affine rotation surface}}
\ENDIF
\end{algorithmic}
\caption{{\tt Affine Rotations-Rat}}\label{alg-affine-rat}
\end{algorithm}

\end{document}